\documentclass[a4paper]{amsproc}
\usepackage{
 amsmath,
 amsxtra,
 amsthm,
 amssymb,
 etex,
 mathrsfs,
 mathtools,
 tikz-cd,
 xr,
 enumitem,
 float,
 url,
 booktabs,
 longtable
 }
\usepackage[T1]{fontenc}
\usepackage{orcidlink}

\theoremstyle{plain}
 \newtheorem{thm}{Theorem}[section]
 \newtheorem{prop}[thm]{Proposition}
 \newtheorem{lem}[thm]{Lemma}
 \newtheorem{cor}[thm]{Corollary}

\theoremstyle{definition}
 \newtheorem{exm}[thm]{Example}
 \newtheorem{dfn}[thm]{Definition}
 \newtheorem*{dfn*}{Definition}

 \newtheorem*{hyp*}{Hypothesis}
 \newtheorem*{hyps*}{Hypotheses}
 \newtheorem{rem}[thm]{Remark}
 
\theoremstyle{remark}
 \newtheorem*{rem*}{Remark}
 
 \newtheorem*{warn*}{Warning}
 \newtheorem*{rems*}{Remarks}
 \newtheorem*{not*}{Notation}
 \newtheorem*{nots*}{Notations}
 \newtheorem*{convs*}{Conventions}
 \newtheorem*{exm*}{Example}
 \newtheorem*{pf*}{Proof}
 \numberwithin{equation}{section}

\renewcommand{\le}{\leqslant}\renewcommand{\leq}{\leqslant}
\renewcommand{\ge}{\geqslant}\renewcommand{\geq}{\geqslant}

\newcommand{\QQ}{\mathbb{Q}}
\newcommand{\ZZ}{\mathbb{Z}}
\newcommand{\CC}{\mathbb{C}}
\newcommand{\RR}{\mathbb{R}}

\newcommand{\cI}{\mathcal{I}}

\newcommand{\p}{\mathfrak{p}}

\newcommand{\Cp}{\mathbb{C}_p}
\newcommand{\Zp}{\mathbb{Z}_p}

\newcommand{\Bf}{\mathbb{F}}
\newcommand{\NN}{\mathbb{N}}

\newcommand{\sopfr}{\text{\it sopfr}\:}

\begin{document}

\title[On a multiplicative non-Hecke twist of motivic L-functions]{On a multiplicative non-Hecke twist of motivic L-functions}

\subjclass[2010]{Primary 11R42; Secondary 11M41, 11R23, 11S80}

\keywords{Euler product, Dirichlet series, $L$-function, Multiplicative character, non-Hecke character, $p$-adic Euler product, $p$-adic Dirichlet series}

\author[Heiko Knospe]{Heiko Knospe \orcidlink{0000-0002-6823-6250}}
\author[Andrzej D\k{a}browski]{Andrzej D\k{a}browski \orcidlink{0000-0002-1524-4871}}

\email{heiko.knospe@th-koeln.de}
\email{andrzej.dabrowski@usz.edu.pl \and dabrowskiandrzej7@gmail.com } 
\address{Institute of Computer and Communication Technology, Technische Hochschule K\"{o}ln, Betzdorfer Str. 2, 50679 K\"{o}ln, Germany}
\address{Institute of Mathematics, University of Szczecin, Wielkopolska 15, 70-451 Szczecin, Poland}

\begin{abstract} We investigate the twisting of motivic $L$-functions 
by a family of multiplicative characters $\psi$, defined on prime ideals $\mathfrak{p}$ via $\psi(\mathfrak{p})=\alpha^{N(\mathfrak{p})}$ for a fixed $\alpha \in \CC$. One can extend
 $\psi$ to a continuous non-Hecke character on the idele group of a number field.  For $|\alpha|<1$, the resulting $\psi$-twisted $L$-function  has interesting  analytic properties: an enhanced half-plane of absolute convergence, preservation of the Euler product structure, and 
 meromorphic continuation to the complex plane. We study the asymptotic behavior as $\alpha\rightarrow 1-$ and give applications to 
  Dirichlet $L$-functions and $L$-functions associated to modular forms. Furthermore, we show that $\psi$-twisting allows a novel construction of convergent $p$-adic Dirichlet series and $p$-adic Euler products which have some similarities with their complex counterparts.
\end{abstract}

\maketitle

  
\section{Introduction}\label{Sect1}

 Let $L(M,s) = \prod_v (P_v(N(\p_v)^{-s}))^{-1}$ 
 be the $L$-function of a pure motive  $M$ of degree $d$ and weight $w$ over a number field $K$. The Euler product is taken over the finite places $v$ of $K$, where $\p_v$ 
 is the associated ideal. For places $v$ of good reduction, the local factor
 $P_v(T)^{-1}$ is  given by the characteristic polynomial $P_v(T)=\det(1-\text{Frob}_v\, T\,  |\, V_l)$, where $\text{Frob}_v$ is the Frobenius acting on an appropriate $l$-adic cohomology space $V_l$ (see \cite{deligne}).
 The polynomial $P_v(T)$ has degree $d$ and its
 inverse roots (the eigenvalues of $\text{Frob}_v$) have absolute value $N(\p_v)^{w/2}$. Standard examples include:
 \begin{enumerate}
 \item The Riemann zeta function, where $P_p(T)=1-T$, giving $d=1$ and $w=0$.
 \item Dirichlet $L$-function for a character $\chi$, where $P_p(T)=1-\chi(p)T$, giving $d=1$ and $w=0$.
 \item $L$-function of a unitary Hecke character $\psi$, where $P_v(T)=1-\psi(\p_v)T$, $d=1$ and $w=0$. 
 \item $L$-function of a newform $f$ of weight $k$, where $P_p(T)=1-a_p(f) T + p^{k-1} T^2$, giving $d=2$ and $w=k-1$.
 \item $L$-function of the Rankin-Selberg product of two newforms of weight $k$, where $P_p(T)$ has degree $d=4$ and $w=2k-2$.
 \end{enumerate}
 The Euler product  converges absolutely for $\Re(s) > 1 + \frac{w}{2}$. It is conjectured (and known in many cases, including the above examples) that $L(M,s)$ admits
 a meromorphic continuation to the entire complex plane and satisfies a functional equation.
 
 This article investigates twists of $L$-functions by a multiplicative character $\psi : \mathcal{I}_K \rightarrow \CC^{\times}$ defined on the group of fractional ideals of $K$.
 A primary objective is to expand the domain of convergence of the Dirichlet series and the Euler product.
 Twisting $L(M,s)$ by $\psi$ means that each local polynomial $P_v(T)$ is replaced with $P_v( \psi(\p_v) T)$. Standard twists, such as
 Dirichlet characters or unitary Hecke characters satisfy $| \psi(\p_v) | = 1$ and thus do not alter the half-plane of  convergence.
 A non-unitary Hecke characters $\psi$, for which 
 $ | \psi(\p_v) | = N(\p_v)^{k} $ for some $k \in \RR$, merely shifts the half-plane of absolute convergence to $\Re(s) > 1 + \frac{w}{2} + k$. To achieve a more significant expansion of the convergence domain, we need a character whose
absolute value  $ | \psi(\p_v) |$ decays faster than any power of $N(\p_v)^{-1}$. Such a character is not a Hecke character,
 and the resulting $L$-function $L(M,s,\psi)$ is not motivic.  Nevertheless, as we will show, $L(M,s,\psi)$ has favourable  analytic properties.  
  
The main results of this article are as follows. We define a new family of multiplicative characters $\psi$ and study the associated twisted $L$-functions.
Theorem \ref{merom} establishes that the $\psi$-twisted $L$-function $L(M,s,\psi)$ has an expanded half-plane of absolute convergence. Furthermore, its Euler product defines a meromorphic function on the complex plane. Theorem \ref{modthm} provides the corresponding result for $L$-functions of newforms.  On the $p$-adic side, we show that $\psi$-twisting yields a convergent $p$-adic Dirichlet series that also admits a convergent Euler product. Finally, Theorem \ref{analytic} gives the Mahler expansion of the $\psi$-twisted $p$-adic series.

\section{Multiplicative non-Hecke characters}
\label{Sect2}
In this section, we define a family of multiplicative twists and analyse their properties.

\begin{dfn}
Let $K$ be a number field, $\cI_K$ its group of fractional ideals and $\alpha \in \CC^*$ a parameter. 
Define a family of multiplicative characters $\psi: \cI_K \rightarrow \CC^*$ by
$$ \psi(\p) = \alpha^{N(\p)} $$
on prime ideals $\p$ 
and extending to all of $\cI_K$ multiplicatively.
\end{dfn}

\begin{rem} This character 
$\psi$ can also be viewed as an unramified character on the group of {\em finite ideles of $K$}. By setting it to be trivial on the infinite places of $K$, i.e., $\psi_{\infty}(x)=\bf{1}$, we obtain a character on the full idele group. However, $\psi$ cannot be turned into a Hecke character (an idele class character). In fact, 
the continuous characters of $\RR^*$ and $\CC^*$ are of the form $x \mapsto sgn(x)^m |x|^{s}$ or $z \mapsto (z/|z|)^n |z|^{s}$, where $m \in \{0,1\}$, $n \in \ZZ$ and $s \in \CC$ (see \cite{neukirch}). 
The character $\psi$ is of exponential type on the finite ideles and cannot be balanced by the infinite part which is necessarily of polynomial type. So triviality on the principal ideles of $K$ cannot be achieved.
\end{rem}

For the remainder of this article, we specialise to $K=\QQ$. Here, ideals are generated by positive integers and 
$\psi$ becomes a completely multiplicative arithmetic function on $\NN$. Note that $\psi(n) \neq \alpha^n$ for composite $n$. The twist  $n \mapsto \alpha^n$ is a character of the {\em additive group} $\ZZ$, and twisting the Riemann zeta function by $\alpha^n$ famously yields Lerch's transcendent (\cite{erdelyi}). The 
special values of Lerch's zeta function at integers $s=m$ are $m$-th order polylogarithms $Li_m(\alpha)$.

Our multiplicative twist is fundamentally different. For $n \in \NN$, we have 
 $\psi(n) = \alpha^{S(n)} $, where $S(n)$ is the {\em sum of prime factors}  of $n$ counted with multiplicity.

\begin{dfn} Let $S: \NN \rightarrow \NN$ be the arithmetic function defined by  $S(1) =0$ and, for $n\geq 2$ with prime factorisation $n=p_1^{k_1} p_2^{k_2} \cdots p_r^{k_r}$, by
$$ S(n) = k_1 p_1 + k_2 p_2 + \dots + k_r p_r .$$
 $S(n)$ is often denoted by  $\sopfr(n)$ ({\em sum of prime factors with repetition}) and known as  {\em integer logarithm of $n$} (see OEIS \cite{oeis} A001414).
 \end{dfn}
 
The first few elements of $S(n)$ for $n=1,2,\dots,10$ are $0, 2, 3, 4, 5, 5, 7, 6, 6, 7$ (see Figure \ref{fig-sopfr}). 

\begin{prop} Let $S(n)=\sopfr(n)$ be the  integer logarithm. Then:
\begin{enumerate}
\item[(a)] $S$ is a completely additive function: $S(n_1 n_2) = S(n_1) + S(n_2)$ for all $n_1,n_2 \in \NN$.
\item[(b)] The character $\psi(n)=\alpha^{S(n)}$ is completely multiplicative.
\item[(c)]  For all $n \in \NN$, we have the lower bound
$S(n) \geq \frac{3}{\log(3)} \log(n) $, with equality if and only if $n=3^k$, $k \geq 0$.
\item[(d)] For any $X \geq 3$, we have $S(n) \geq \frac{X}{\log(X)} \log(n)$ for all $n \in \NN$ whose prime factors $p$ satisfy $p \geq X$.  
\end{enumerate}
\label{f}
\end{prop}
\begin{proof} (a) and (b) are clear from the definitions. Now using (a) and the additivity of the $\log$-function, it is sufficient to show (c) and (d) for {prime} numbers $p$. In this case, we have $S(p)=p$. 
We analyse the function $f(x)=\frac{x}{\log(x)}$ for $x>1$. The derivative shows that $f(x)$ is increasing for  $x > e$. For primes numbers $p$, the minimum value is at $p=3$ since $f(2) > f(3)$. This implies $\frac{p}{\log(p)} \geq \frac{3}{\log(3)}$ for all primes $p$, with equality if and only if $p=3$. This proves (c). Part (d) follows from the same argument, using the fact that $\frac{p}{\log(p)} \geq \frac{X}{\log(X)}$ for all primes $p \geq X \geq 3$.
\end{proof}

\begin{figure}[h]
\includegraphics[width=8cm]{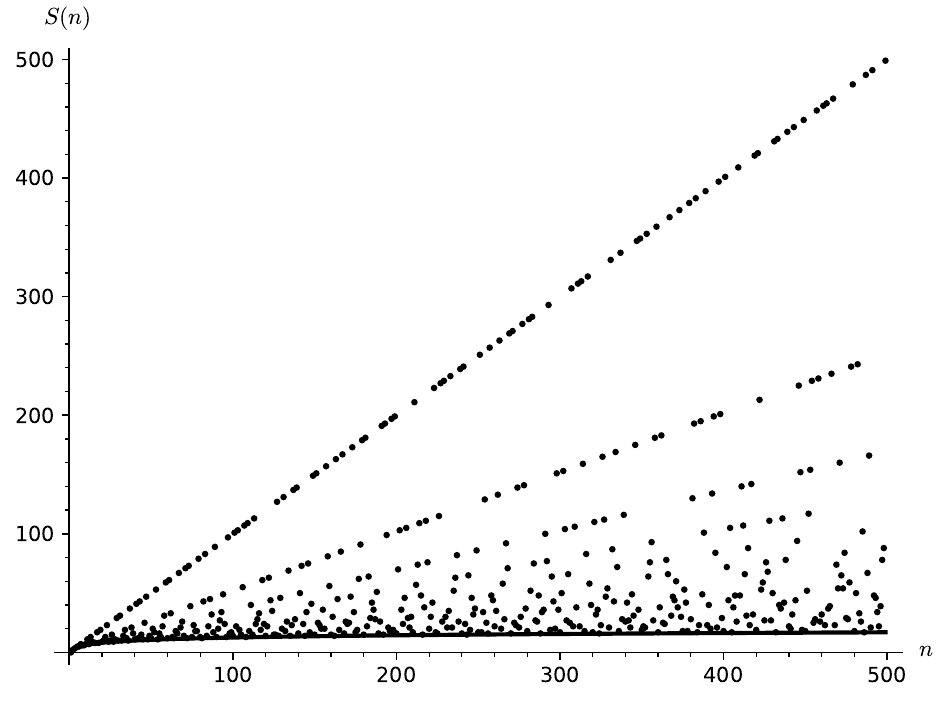}
\caption{Graph of the integer log function $S$. The diagonal dotted lines are given by the values of $S$ at $p$, $2p$, $3p$, $\dots$, where $p$ is a prime number. The graphic also shows the lower bound $\frac{3}{\log(3)} \log(n)$.}
\label{fig-sopfr}
\end{figure}   

\begin{rem}
The average order of $S(n)$ is known. Jakimcyuk \cite{jak} shows the 
asymptotic formula
$$\sum_{i=1}^n S(i) \sim \frac{\pi^2}{12} \frac{n^2}{\log(n)} .$$
\end{rem}

\begin{dfn} 
The set of preimages  $S^{-1}(\{m\})$ corresponds bijectively to the set of {\em partitions of $m$ into prime parts}. We denote the number of such partitions by $\vartheta(m)$.
 \end{dfn}
 For example, $\vartheta(7)=3$, as the partitions into prime parts are $\{7\}$, $\{5,2\}$, $\{3,2,2\}$ . One can show that $\vartheta(m)=1$ for $m=2,3,4$ and $\vartheta(m) \geq 2$ for $m\geq 5$. 
\begin{dfn} 
The $\psi$-twist of  an arithmetic function $f$ is given by point-wise multiplication:
$$    (\psi f)(n) =  \psi(n) \cdot f(n) = \alpha^{S(n)} \cdot f(n) .$$
\end{dfn}

\begin{prop} 
 Let $R$ be the ring of arithmetic function with point-wise addition and Dirichlet convolution. The map $f \mapsto \psi f$  is a ring isomorphism of $R$.
If $f$ is invertible, i.e., if $f(1) \neq 0$,  then the Dirichlet inverse of $ \psi f $ is $\psi f^{-1}$. In particular, if $f$ is completely multiplicative, then $( \psi f)^{-1} = \psi \mu f$, where $\mu$ is the Möbius function.
\end{prop}
\begin{proof} This is a standard result for the point-wise multiplication with a completely multiplicative function (see \cite{apostol}, chapter 2). 
\end{proof}

\section{$\psi$-Twists of $L$-functions}

We apply $\psi$-twisting to Dirichlet series and $L$-functions and analyse the properties of the resulting functions.

\subsection{Twists of Dirichlet Series}

\begin{dfn} 
Let $c(n)$ be an arithmetic function and $\sum^{\infty}_{n=1} c(n) n^{-s}$ the associated formal Dirichlet series.
For a non-zero parameter $\alpha \in \CC$, the corresponding $\psi$-twisted formal Dirichlet series is 
$$ \sum^{\infty}_{n=1} \psi(n) c(n) n^{-s} = \sum^{\infty}_{n=1} \alpha^{S(n)} c(n) n^{-s} .$$
\end{dfn}

\begin{rem} 
The twisted series can be viewed as a function in two variables, $s$ and $\alpha$.
For a fixed $s$ it is power series in $\alpha$. Rearranging the sum  gives
\begin{equation} 
\begin{aligned}
& \sum^{\infty}_{m=0}   \left( \sum_{n \in S^{-1}(\{m\})}^{\phantom{2}}   c(n) n^{-s} \right) \alpha^{m} \\ 
& = \  c(1) + (c(2)2^{-s}) \alpha^2 +  (c(3)3^{-s}) \alpha^3  +  (c(4)4^{-s}) \alpha^4  + \left(c(5)5^{-s}+c(6)6^{-s}\right)  \alpha^5 + \dots 
\end{aligned}
\label{alphaseries}
\end{equation}
The coefficient of $\alpha^m$ is 
a sum over all integers $n$ whose integer logarithm $S(n)$ is $m$. The number of such integers is $\vartheta(m)$, the number of partitions of $m$ into prime parts.
\end{rem}

\begin{exm}
Consider the $\psi$-twisted Riemann zeta function at $s=0$. Then $c(n)n^{-s} =1$ for all $n \in \NN$ and the coefficient of $\alpha^m$ is $\vartheta(m)$.  
The expansion has a well-known generating function:
$$ 1+ \sum^{\infty}_{m=2} \vartheta(m) \alpha^m  = \prod_{ p \text{ prime}} \frac{1}{1-\alpha^p}$$
\end{exm}

For $|\alpha|<1$, the twist introduces a strong convergence factor. The half-plane of convergence is expanded, as the following proposition shows.
\begin{prop} Let $ \sum^{\infty}_{n=1} c(n) n^{-s} $ be a Dirichlet series with  abscissa of absolute convergence $\sigma_a^0$, and let $|\alpha| \leq 1$. The
abscissa of absolute convergence $\sigma_a$ of the $\psi$-twisted Dirichlet series $\sum^{\infty}_{n=1} \alpha^{S(n)} c(n) n^{-s}$ satisfies
 
$$\sigma_a \leq  \sigma_a^0 + \frac{3}{\log(3)} \log(|\alpha|). $$
\label{conv}
\end{prop}
 \begin{proof} From Proposition \ref{f}
 we have $|\alpha|^{S(n)} \leq |\alpha| ^{  \frac{3}{\log(3)} \log(n)} = n^{ \frac{3}{\log(3)} \log |\alpha| }$. This shows the claim.
 \end{proof}

If the coefficients are {\em multiplicative}, i.e., if $c(mn)=c(m)c(n)$ for $\gcd(m,n)=1$, the $\psi$-twisted series also has an Euler product:
\begin{equation}   \sum^{\infty}_{n=1} \alpha^{S(n)} c(n) n^{-s} = \prod_p \sum_{m=0}^{\infty}  \alpha^{S(p^m)} c(p^m)\: (p^m)^{-s} .
\label{multiplicative}
\end{equation}
Since $S(p^m)=mp$, this can be rewritten as 
$$ \prod_p \left( \sum_{m=0}^{\infty} c(p^m)\: (\alpha^p p^{-s})^m \right) . $$
This can be leveraged to find the exact region of convergence and to establish a meromorphic continuation.
We choose $X \geq 3$ and split this into a product over primes  $p<X$ and $p \geq X$, respectively, and expand the second product as a Dirichlet series to obtain
 \begin{equation}
 \left( \prod_{p<X} \sum_{m=0}^{\infty}  c(p^m)\: (\alpha^p p^{-s})^{m} \right) \left( \sum_{\substack{n\geq 1\\ p\mid n \Rightarrow p\geq X}}
  \alpha^{S(n)} c(n)\: n^{-s} \right) .
\label{second}
\end{equation}
The second factor of (\ref{second}) converges absolutely in a larger half-plane $\Re(s) \geq \sigma_a^0 + \frac{X}{\log(X)} \log |\alpha | $. To this end, we note that $S(n) \geq \frac{X}{\log(X)} \log(n) $ by Proposition \ref{f}(c), and hence
$$|\alpha|^{S(n)} \leq |\alpha| ^{  \frac{X}{\log(X)} \log(n)} = n^{ \frac{X}{\log(X)}  \log|\alpha |} .$$
Since $\frac{X}{\log(X)} \rightarrow \infty$ as $X \rightarrow \infty$, this process extends the function meromorphically to the entire complex plane,
provided that the first factor is meromorphic.

\begin{thm} Let $L(M,s) = \prod_p (P_p( p^{-s}))^{-1}$ 
 be the $L$-function of a pure motive  of degree $d$ and weight $w$ over $\QQ$. Assume for simplicity that for primes $p$ of bad reduction, the inverse roots of $P_p(T)$ of the primes of bad reduction have absolute value at most $p^{w/2}$, and that $p=3$ is a prime of good reduction. Let $|\alpha|<1$.
 Then the abscissa of convergence $\sigma_c$ and the abscissa of absolute convergence $\sigma_a$  of the twisted $L$-function $L(M,s,\psi)$ are
 $$ \sigma_c = \sigma_a =  \frac{3}{\log(3)} \log(|\alpha|) + \frac{w}{2} . $$
The function $L(M,s,\psi)$  extends to a meromorphic function on $\CC$ that has no zeroes. Its poles correspond to the
zeroes of the local factors $P_p( \alpha^p p^{-s})$. 
 Each prime $p$ of good reduction yields a family of poles whose real part is
 $$ \frac{p}{\log(p)} \log |\alpha| + \frac{w}{2} $$
and whose imaginary parts are $\frac{2\pi }{\log(p)}$-periodic .  
\label{merom}
 \end{thm}
 
 \begin{proof} Let $p$ be a prime of good reduction. The local polynomial $P_p( p^{-s})$ factors as $\prod_{i=1}^d (1-c_{p,i}\, T)$,
 where the inverse roots have absolute value $|c_{p,i}| = p^{w/2}$. The twisted factor is 
 $P_p(\psi (p) p^{-s}) = \prod_{i=1}^d (1- \alpha^{p} c_{p,i} \, p^{-s})$. 
The Euler product converges absolutely if and only if the series $\sum_{p} \sum_{i=1}^d  \alpha^{p} c_{p,i} \: p^{-s}$ converges absolutely and all terms $\alpha^{p} c_{p,i} \: p^{-s}$  are $\neq 1$.
Let $\sigma= \Re(s)$. The absolute values of the terms are
$$ | \alpha^{p} c_{p,i} \: p^{-s} | = e^{p \log |\alpha| + \frac{w}{2} \log(p)- \sigma \log(p)} = e^{p\left(\log |\alpha| + (\frac{w}{2}- \sigma) \frac{\log(p)}{p}\right)} .$$
If $\sigma > \frac{w}{2} + \frac{3}{\log(3)} \log|\alpha|$ then we have $\log |\alpha| + (\frac{w}{2} - \sigma) \frac{\log(p)}{p}<0$. Since $\frac{\log(p)}{p}$ converges to $0$, the series $\sum_{p} | \alpha^{p} c_{i,p}\, p^{-s} |$ 
is dominated by a convergent geometric series. Hence the infinite product converges absolutely in the half-plane $\sigma >  \frac{3}{\log(3)} \log|\alpha| + \frac{w}{2}$.

The assumption on the primes of bad reduction ensures that these primes cannot contribute any poles having a larger real part. Furthermore, since we assumed that $p=3$ is a prime of good reduction, poles on the line $\sigma = \frac{3}{\log(3)} \log|\alpha| + \frac{w}{2}$ exist, so it is the true abscissa. 
 The meromorphic extension to $\CC$ follows from the factorisation method in equation (\ref{second}). To this end, we note that in our situation the first factor of  (\ref{second}) is a finite product of the meromorphic functions $(P_p( p^{-s}))^{-1}$.
 Finally, the poles of $L(M,s,\psi)$ are the solutions to
 $\alpha^{p} c_{p,i}\: p^{-s} = 1$ which yields the stated result.
  \end{proof}
 
\begin{figure}[h]
\includegraphics[width=8cm]{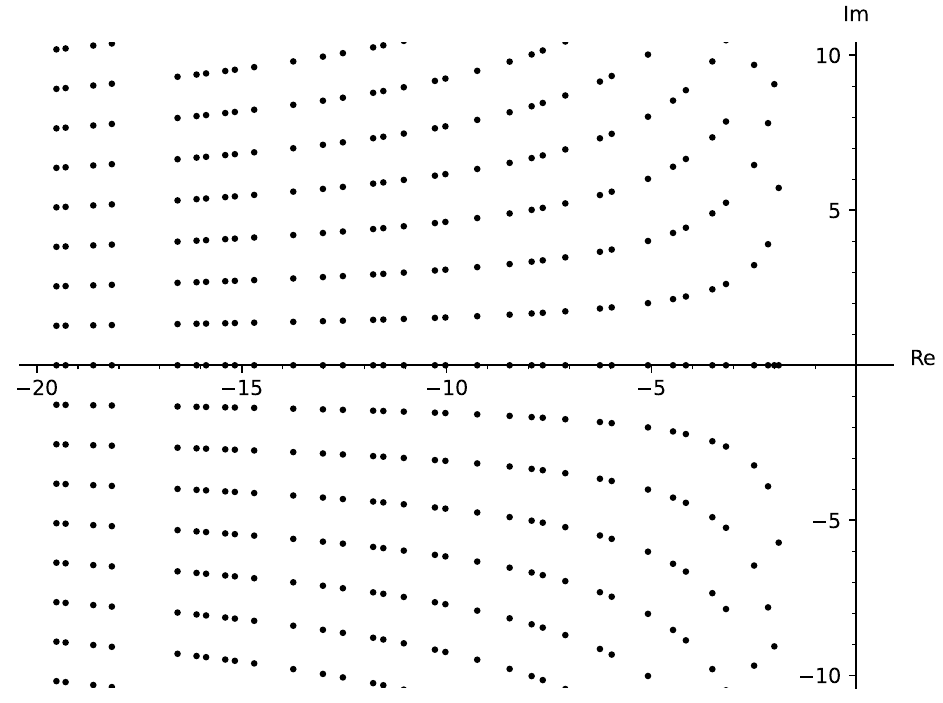}
\caption{Poles of the $\psi$-twisted Riemann zeta function for $\alpha=\frac{1}{2}$. The largest real parts of poles are at $\frac{-3\log(2)}{\log(3)} \approx -1.89$. }
\label{pospoles0}
\end{figure}

 It is worth noting that  the abscissa $\sigma_a$ converges to  $\frac{w}{2}$ as $|\alpha| \rightarrow 1-$, and not to $1+  \frac{w}{2}$ as one might expect.

\begin{rem} The proof of Theorem \ref{merom} shows that the {\em Euler product} (over all $p$) defines a meromorphic function on the complex plane. However, 
the associated {\em Dirichlet series} converges only for $\Re(s)=\sigma > \frac{3}{\log(3)} \log|\alpha| + \frac{w}{2}$. 
\end{rem}

The following theorem shows that the Dirichlet series representing the second factor of (\ref{second}) 
has an integral representation as a Mellin-transform. 
This can be used to extend the half-plane of convergence and provides another approach to the meromorphic continuation of $L(M,s,\psi)$.

\begin{thm}\label{thm:mellin}
Let $L(M,s,\psi)=\sum_{n=1}^{\infty} \alpha^{S(n)} c(n) n^{-s}$ be the $\psi$-twisted $L$--function of a pure motive $M$ over $\QQ$, and suppose that
$$ |c(n)| \leq C n^{k} $$
for some constants $C>0$ and $k\in\RR$. Fix an integer $m\geq 0$ and a real number $X\geq 3$ such that
$$ \frac{X}{\log X}\,\log|\alpha| + k < -\,m-1 . $$
Define
$$ \phi_X(t) \;=\; \sum_{\substack{n\geq 1\\ p\mid n \Rightarrow p\geq X}} \alpha^{S(n)}\,c(n)\,e^{-nt}
\qquad (t>0). $$
Then:

\begin{enumerate}
\item[(a)] The series defining \(\phi_X(t)\) and its derivatives
$$ \phi_X^{(i)}(t) \;=\; \sum_{\substack{n\geq 1\\ p\mid n \Rightarrow p\geq X}} (-n)^i \alpha^{S(n)}\,c(n)\,e^{-nt},
\qquad 0\leq i\leq m,
$$
converge absolutely and uniformly on every compact subset of $(0,\infty)$. 
Moreover, the bounds on the coefficients imply that the series defining $\phi_X^{(i)}(0)$ converge absolutely and that
$\phi_X^{(i)}(t) \rightarrow \phi_X^{(i)}(0)$ as $t \rightarrow 0+$. In particular, 
 $\phi_X\in C^m([0,\infty))$. 

\item[(b)]  For each \(i=0,1,\dots,m\), the Mellin transform
$$ L(\phi_X^{(i)},s)
\;=\;
\frac{1}{\Gamma(s)}\int_0^{\infty} \phi_X^{(i)}(t)\,t^{s}\,\frac{dt}{t}
$$
defines a holomorphic function on the half-plane $\Re(s)>0$.

\item[(c)]  For $\Re(s)$ sufficiently large we have the identity
$$
L(\phi_X,s)
\;=\;
\sum_{\substack{n\geq 1\\ p\mid n \Rightarrow p\geq X}} \alpha^{S(n)}\,c(n)\,n^{-s}
\;=:\;
L_X(M,s,\psi),
$$
i.e., $L(\phi_X,s)$ is the $\psi$-twisted $L$-function of $M$ with the Euler factors at primes $p<X$ removed. Moreover, for all $s$ with $Re(s)>0$ and all $0\leq i\leq m$,
$$ L(\phi_X,s)
= -\,L(\phi_X',s+1)
= \dots
= (-1)^i\,L(\phi_X^{(i)},s+i). $$

\item[(d)] The identities in \emph{(c)} extend $L_X(M,s,\psi)$ holomorphically to the half-plane $Re(s)>-m$. Since $L(M,s,\psi)$ differs from $L_X(M,s,\psi)$ by a finite product of local factors at $p<X$, this yields a meromorphic continuation of $L(M,s,\psi)$ to the half-plane $\Re(s)>-m$.
\end{enumerate}
\end{thm}
\begin{proof} We use the classical representation of Dirichlet series as Mellin transforms (cf.\ Zagier \cite{zagier}). Since $\phi_X(t)$ lacks a full asymptotic 
expansion at $t = 0$ (see Remark \ref{rem:nonasymp}), we perform repeated integration 
by parts up to order $m$ to establish the shift relations of the derivatives and the analytic continuation.

By Proposition~\ref{f}(d), for every $n$ whose prime factors satisfy $p\geq X$ we have
$$ S(n) \;\geq\; \frac{X}{\log X}\,\log n,
$$
hence
$$
|\alpha|^{S(n)}
\;\leq\;
|\alpha|^{\frac{X}{\log X}\log n}
=
n^{\frac{X}{\log X}\log|\alpha|}.
$$
Using the assumed bound $|c(n)|\leq C n^{k}$, we obtain
$$ |\alpha^{S(n)} c(n)|
\;\leq\;
C\,n^{\frac{X}{\log X}\log|\alpha| + k}. $$
Set
$$ A := \frac{X}{\log X}\,\log|\alpha| + k.$$
The hypothesis $A < -m-1$ gives for each integer $0\leq i\leq m$
$$ A+i \;<\; -1.$$

For $t>0$ we have
$$ \phi_X(t)
=
\sum_{\substack{n\geq 1\\ p\mid n \Rightarrow p\geq X}}
\alpha^{S(n)} c(n) e^{-nt}, $$
and termwise differentiation yields
$$ \phi_X^{(i)}(t)
=
\sum_{\substack{n\geq 1\\ p\mid n \Rightarrow p\geq X}}
(-n)^i \alpha^{S(n)} c(n) e^{-nt},
\qquad 0\leq i\leq m.$$

The series for \(\phi_X^{(i)}(t)\) converges absolutely and uniformly on compact intervals in $(0,\infty)$, which implies \(\phi_X\in C^m((0,\infty))\) and permits termwise differentiation. 
Rapid decay as $t \to\infty$ follows from the factor $e^{-nt}$ in each term.
To analyse $t\to 0+$, note that
$$ \phi_X^{(i)}(t)
=
\sum_{\substack{n\geq 1\\ p\mid n \Rightarrow p\geq X}} (-n)^i \alpha^{S(n)} c(n) e^{-nt}
\;\xrightarrow[t\to 0+]{}\;
\sum_{\substack{n\geq 1\\ p\mid n \Rightarrow p\geq X}} (-n)^i \alpha^{S(n)} c(n), $$
and this limit series converges absolutely since $i+A<-1$ implies
$$ \sum_{n\geq 1} n^{i+A} < \infty. $$
Thus each $\phi_X^{(i)}$ extends continuously to $t=0$. This shows part (a).

For $\Re(s)>0$, the Mellin transform of $\phi_X^{(i)}$ is given by
$$ \int_0^{\infty} \phi_X^{(i)}(t)\, t^{s}\,\frac{dt}{t} = \int_0^{1} \phi_X^{(i)}(t)\, t^{s}\,\frac{dt}{t} + \int_1^{\infty} \phi_X^{(i)}(t)\, t^{s}\,\frac{dt}{t} .
$$

The first integral converges absolutely because $\phi_X^{(i)}$ is bounded near $0$ and $\Re(s)>0$, while the second converges absolutely by rapid decay of $\phi_X^{(i)}$ at infinity. Now standard arguments show that it defines a holomorphic function of $s$ in this half-plane. Dividing by $\Gamma(s)$ preserves holomorphy, proving (b).

For $\Re(s)$ sufficiently large (depending on $A$ and $k$), we may interchange sum and integral by absolute convergence to obtain

$$ \begin{aligned}
L(\phi_X,s)
&=
\frac{1}{\Gamma(s)}\int_0^{\infty}\phi_X(t)\, t^{s}\,\frac{dt}{t} \\
&=
\frac{1}{\Gamma(s)}
\int_0^{\infty}
\left(
\sum_{\substack{n\geq 1\\ p\mid n \Rightarrow p\geq X}}
\alpha^{S(n)} c(n) e^{-nt}
\right)
t^{s}\,\frac{dt}{t} \\
&=
\frac{1}{\Gamma(s)}
\sum_{\substack{n\geq 1\\ p\mid n \Rightarrow p\geq X}}
\alpha^{S(n)} c(n)
\int_0^{\infty} e^{-nt}\, t^{s}\,\frac{dt}{t} \\
&=
\frac{1}{\Gamma(s)}
\sum_{\substack{n\geq 1\\ p\mid n \Rightarrow p\geq X}}
\alpha^{S(n)} c(n) n^{-s}\Gamma(s) = L_X(M,s,\psi) .
\end{aligned}$$
By holomorphy of both sides on $\Re(s)>0$, the identity extends to the whole half-plane $\Re(s)>0$, proving the first claim in part (c).

An integration by parts yields, for $\Re(s)>0$,
$$
\int_0^{\infty} \phi_X'(t)\, t^{s+1}\,\frac{dt}{t}
=
\left[\phi_X(t)\, t^{s}\right]_{0}^{\infty}
-
s \int_0^{\infty} \phi_X(t)\, t^{s}\,\frac{dt}{t}.
$$
The boundary term vanishes: at $\infty$ by rapid decay of $\phi_X(t)$, and at $0$ because $\phi_X(t)$ is bounded and $\Re(s)>0$. Hence
$$
L(\phi_X',s+1)
=
\frac{1}{\Gamma(s+1)}\int_0^{\infty} \phi_X'(t)\, t^{s+1}\,\frac{dt}{t}
=
-\frac{s}{\Gamma(s+1)}\int_0^{\infty} \phi_X(t)\, t^{s}\,\frac{dt}{t}
=
-\,L(\phi_X,s),
$$
using $\Gamma(s+1)=s\Gamma(s)$. Iterating this argument gives, for $0\leq i\leq m$,
$$
L(\phi_X^{(i)},s+i)
=
(-1)^i L(\phi_X,s),
$$
which is the chain of equalities in (c).

Each function $L(\phi_X^{(i)},s)$ is holomorphic for $\Re(s)>0$ by (b). The shift relations in (c) show that $L(\phi_X,s)$ admits analytic continuation to the half-plane $\Re(s)>-m$ via
$$
L(\phi_X,s)
=
(-1)^m L(\phi_X^{(m)},s+m),
$$
which is holomorphic there because $\Re(s+m)>0$. As $L_X(M,s,\psi)=L(\phi_X,s)$ on a right half-plane and the two sides are holomorphic on $\Re(s)>-m$, this gives a holomorphic continuation of $L_X(M,s,\psi)$ to $\Re(s)>-m$. Finally, $L(M,s,\psi)$ differs from $L_X(M,s,\psi)$ only by a finite product of local Euler factors at primes $p<X$, hence $L(M,s,\psi)$ extends meromorphically to $\Re(s)>-m$, proving part (d).
\end{proof}

\begin{rem} Any chosen positive integer $m$ satisfying the above theorem can be achieved by increasing $X$, i.e., by removing Euler factors, or by decreasing $\alpha$.
If $\frac{3}{\log 3}\,\log|\alpha| + k < -\,m-1$ then it not necessary to remove any Euler factor to obtain a holomorphic continuation to $\Re(s)>-m$.
\end{rem} 

\begin{rem} In contrast to many classical Dirichlet series, the derivatives $\phi_X^{(i)}(t)$ diverge as $t \rightarrow 0+$ if $i>m$, due to the family of poles in the left half-plane (see Theorem \ref{merom}). In particular, $\phi_X$ is a $C^{\infty}$ function on $(0,\infty)$, but only $C^m$ on $[0,\infty)$. Furthermore,  $\phi_X$ does not have a Taylor expansion at $t=0+$, where the function has an essential singularity.

\label{rem:nonasymp}
\end{rem}

\begin{cor} Let $\phi_X^{(i)}(t)$ and $L(\phi_X,s)$ be as in Theorem \ref{thm:mellin}. As above, let $0 \leq i \leq m$. Then
$$ L(\phi_X, -i) = (-1)^i \phi_X^{(i)}(0) = (-1)^i \sum_{\substack{n\geq 1\\ p\mid n \Rightarrow p\geq X}} \alpha^{S(n)} c(n) (-n)^i = \sum_{\substack{n\geq 1\\ p\mid n \Rightarrow p\geq X}} \alpha^{S(n)} c(n) n^i.$$
\end{cor}
Note that the above series converges for $0 \leq i \leq m$ and yields the expected result $L_X(M,-i,\psi)$. 

\subsection{Asymptotic behaviour as $\alpha$ approaches $1$}
The following proposition shows that the twisted Dirichlet series tends towards the original series as $\alpha \in \CC$ approaches $1$ from the left from within a {\em Stolz sector}. 
 
 \begin{prop} Let $\sigma_a^0$ be the abscissa of absolute convergence of a Dirichlet series $ \sum^{\infty}_{n=1} c(n) n^{-s} $. For $\Re(s)>\sigma_a^0$ we have 
 $$ \lim_{\alpha \rightarrow 1-} \left( \sum^{\infty}_{n=1} \alpha^{S(n)} c(n) n^{-s} \right) =  \sum^{\infty}_{n=1} c(n) n^{-s} . $$
 \label{limit}
 \end{prop}
 \begin{proof} For $\Re(s)>\sigma_a^0$ and $|\alpha|<1$, the original Dirichlet series and the twisted series converge absolutely. Hence both series can be re-ordered and expanded in powers of $\alpha$ (see equation (\ref{alphaseries})). Now the claim follows from Abel's theorem.
 \end{proof}

An interesting question is whether the above result can be extended to the critical strip, and in particular to the {critical line} $\Re(s) = \frac{w+1}{2}$. In contrast to $L(M,s)$, the $\psi$-twisted Euler product converges absolutely for all $|\alpha|<1$  on a half-plane that contains the critical strip.

We can re-order $L(M,s,\psi)$ with respect to powers of $\alpha$ and obtain
$$ L(M,s,\psi) = \sum^{\infty}_{m=0}   \left( \sum_{n \in S^{-1}(\{m\})}^{\phantom{2}}   c(n) n^{-s} \right) \alpha^{m} .$$
In order to apply Abel's theorem, the convergence of $ \sum^{\infty}_{m=0}   \left( \sum_{n \in S^{-1}(\{m\})}^{\phantom{2}}   c(n) n^{-s} \right)$ is required.
Due to the re-ordering, however,  this is not obvious if the Dirichlet series converges only conditionally. Furthermore, the limit value is not necessarily $L(M,s)$.

We use an alternative approach based on the Euler product expansion. 
However, Abel's theorem cannot be applied directly to infinite products, and Hardy \cite{hardy} has even given a counterexample for the case of conditional convergence.
As explained in ibid., we need additional assumptions on the coefficients.

\begin{prop} 
Let $\prod_p P_p(p^{-s})^{-1} = \prod_p \prod_{i=1}^d (1- c_{p,i} \, p^{-s})^{-1}$ be the Euler product expansion of $L(M,s)$, and let $s$ be a point on or to the right of the critical line, i.e., $\Re(s) \geq \frac{1+w}{2}$. Suppose that the following conditions are satisfied:
\begin{enumerate} 
\item[(a)] The linear prime series 
\[
\sum_p \left( \sum_{i=1}^d c_{p,i} \right) p^{-s}
\]
converges (at least) conditionally.
\item[(b)] If $\Re(s) = \frac{1+w}{2}$ (i.e., when $s$ lies on the critical line), in addition to (a), the quadratic prime series
\[
\sum_p \left( \sum_{i=1}^d c_{p,i}^2 \right) p^{-2s}
\]
converges conditionally.
\end{enumerate}
Then the limit of $L(M,s,\psi)$ as $\alpha \to 1-$ exists, converges to the classical Euler product, and is non-zero:  
\begin{equation}
\lim_{\alpha \to 1-} L(M,s,\psi) = \prod_p P_p(p^{-s})^{-1}.
\label{convlimit}
\end{equation}
Here, $\alpha$ approaches $1$ from within a Stolz sector (e.g., along the real axis from the left). 
Furthermore, if $K$ is a compact subset of the half-plane $\Re(s) > \frac{1+w}{2}$ such that (a) is satisfied for all $s \in K$, then the convergence as $\alpha \to 1-$ is uniform on $K$. 
\label{limitcrit}
\end{prop}
 
\begin{proof}
We prove the claim by taking the logarithm of both sides. Since the abscissa of absolute convergence of $L(M,s,\psi)$ satisfies $\sigma_a < \frac{1+w}{2}$ for $|\alpha| < 1$ (see Theorem \ref{merom}), the left-hand side converges absolutely for $|\alpha| < 1$, and its logarithm is given by
\begin{equation}
\sum_p \sum_{i=1}^d -\log(1 - \alpha^p c_{p,i} \, p^{-s}) = \sum_p \sum_{k=1}^{\infty} \frac{1}{k} \sum_{i=1}^d \alpha^{kp} c_{p,i}^k \, p^{-ks}.
\label{lhs}
\end{equation}
The formal logarithm of the right-hand side expands as
\begin{equation} 
\sum_p \sum_{i=1}^d -\log(1 - c_{p,i} \, p^{-s}) = \sum_p \sum_{i=1}^d c_{p,i} \, p^{-s} + \frac{1}{2} \sum_p \sum_{i=1}^d c_{p,i}^2 \, p^{-2s} + \sum_p \sum_{k=3}^{\infty} \frac{1}{k} \sum_{i=1}^d c_{p,i}^k \, p^{-ks}.
\label{rhs}
\end{equation}
For the higher-order terms ($k \ge 3$), we have $|c_{p,i}^k p^{-ks}| = p^{kw/2 - k(1+w)/2} = p^{-k/2} \le p^{-3/2}$. Hence
the series of higher-order terms on the right-hand side of (\ref{rhs}) converges absolutely to a finite value. Furthermore, because $|\alpha^{kp}| \le 1$ for all $|\alpha| \le 1$, the Weierstrass $M$-test implies that the $\alpha$-dependent higher-order series ($k\geq 3$)  in (\ref{lhs}) converges absolutely and uniformly for all $|\alpha| \le 1$. Consequently, its limit as $\alpha \to 1-$ equals the corresponding higher-order terms of the right-hand side of (\ref{rhs}).
The same absolute and uniform convergence holds for the quadratic terms ($k=2$) whenever $\Re(s) > \frac{1+w}{2}$, since then $|c_{p,i}^2 p^{-2s}| = p^{-(2\Re(s)-w)}$ with $2\Re(s) - w > 1$.

On the critical line $\Re(s) = \frac{1+w}{2}$, condition (b) guarantees the conditional convergence of the quadratic terms in (\ref{rhs}). Writing $\alpha^{2p} = u^p$ with $u = \alpha^2 \to 1-$, Abel's limit theorem for power series yields
$$ \lim_{\alpha \to 1-} \left( \frac{1}{2} \sum_p \sum_{i=1}^d \alpha^{2p} c_{p,i}^2 \, p^{-2s} \right) = \frac{1}{2} \sum_p \sum_{i=1}^d c_{p,i}^2 \, p^{-2s}.
$$
Similarly, for the linear terms ($k=1$), condition (a) and Abel's limit theorem imply
$$ \lim_{\alpha \to 1-} \left( \sum_p \sum_{i=1}^d \alpha^{p} c_{p,i} \, p^{-s} \right) = \sum_p \sum_{i=1}^d c_{p,i} \, p^{-s}.
$$
Combining the limits for $k=1$, $k=2$, and $k \ge 3$ establishes the identity for the logarithms. Exponentiating both sides yields (\ref{convlimit}) and ensures that the limit value is non-zero. Finally, uniform convergence on a compact set $K \subset \{\Re(s) > \frac{1+w}{2}\}$ follows from the absolute and uniform convergence of the higher-order terms ($k \geq 2$) and Abel's uniform convergence test for the linear term ($k=1$).
\end{proof}
 
\begin{rem} 
 Consider a primitive Dirichlet character $\chi$. Then conditions (a) and (b) of Proposition \ref{limitcrit} follow from established conjectures, as we discuss below.
 
 The following formal identity (with absolute convergence for $\Re(s)>1$)
follows from the M\"obius inversion formula (see \cite{froeberg}):
 \begin{equation} 
\sum_p \chi(p)p^{-s} = \sum_{m=1}^{\infty} \frac{\mu(m)}{m} \log L(ms, \chi^m).
\label{moebiusinv}
\end{equation}
For $\Re(s) > \frac{1}{2}$, the convergence (i.e., condition (a) of  Proposition \ref{limitcrit})  is implied by the {\em Generalized Riemann Hypothesis} (GRH), which ensures that $L(ms, \chi^m) \neq 0$ for all $m \ge 1$, 
and the fact that for $m \ge 2$, the terms $\frac{\mu(m)}{m} \log L(ms, \chi^m)$ decay exponentially which gives the absolute convergence of the tail in the right-hand side of (\ref{moebiusinv}). Applying Perron's formula to the $\chi$-twisted Mangoldt function then yields (under GRH) conditional convergence of the  prime series and equality  (\ref{moebiusinv}).

On the critical line $\Re(s) = \frac{1}{2}$, the situation is more delicate. Conditions (a) and (b) are satisfied if $L(s, \chi) \neq 0$, $L(2s, \chi^2)$ regular and $\neq 0$, and the {\em Deep Riemann Hypothesis} (DRH) holds (see \cite{kconrad, koyama-kurokawa}). In particular, at the central point $s = \frac{1}{2}$, condition (b) requires that $L(1, \chi^2) \neq \infty$, which holds if and only if $\chi$ is a non-quadratic character. For real quadratic characters $\chi$ at $s = \frac{1}{2}$, the pole of $L(1, \chi^2)$ causes the quadratic sum to diverge, driving the $\psi$-twisted product to zero. In particular, the anomalous $\sqrt{2}$ factor observed in certain Euler products (cf.\ \cite{goldfeld, kconrad}) does not arise under the smooth 
$\alpha$-twist.
\label{limitcritrem}
\end{rem}

\begin{rem}
Can we hope to find a functional equation for $\psi$-twisted $L$-functions? This seems very unlikely given the family of poles of $L(M,s,\psi)$  in the left half-plane (see Theorem \ref{merom} and Figure \ref{pospoles0}), which cannot be compensated for by a finite number of $\Gamma$-factors. 

\end{rem}

\subsection{$\psi$-twists of Dirichlet $L$-functions}
For  a Dirichlet character $\chi$, the motive has weight $w=0$ and degree $d=1$. Theorem \ref{merom} specializes to:

\begin{cor}
Let $\chi$ be a Dirichlet character  and  $|\alpha| \leq 1$. Let $\sigma_c$ and $\sigma_a$ be the abscissas of convergence of the $\psi$-twisted Dirichlet $L$-function 
$  L(s, \psi\chi)$. If $\chi(3)\neq 0$ then
$$\sigma_c = \sigma_a = \frac{3}{\log(3)} \log(|\alpha|) < 0 .$$
If $\chi(3)=0$ then the abscissa is determined by the smallest prime $q$ not dividing the conductor of $\chi$, i.e., 
$\sigma_c = \sigma_a = \frac{q}{\log(q)} \log(|\alpha|)$. The function $ L(s,\psi \chi) $ extends to a meromorphic function on $\CC$ without zeroes
and with simple poles at
$$ \frac{p \log(\alpha)}{\log(p)} + \frac{\log(\chi(p))}{\log(p)}  + \frac{ k }{\log(p)} \, 2\pi i  , $$
where $p$ is a prime with $\chi(p) \neq 0$ and $k \in \ZZ$. 
\label{mero}
\end{cor}

\begin{figure}[h]
\includegraphics[width=8cm]{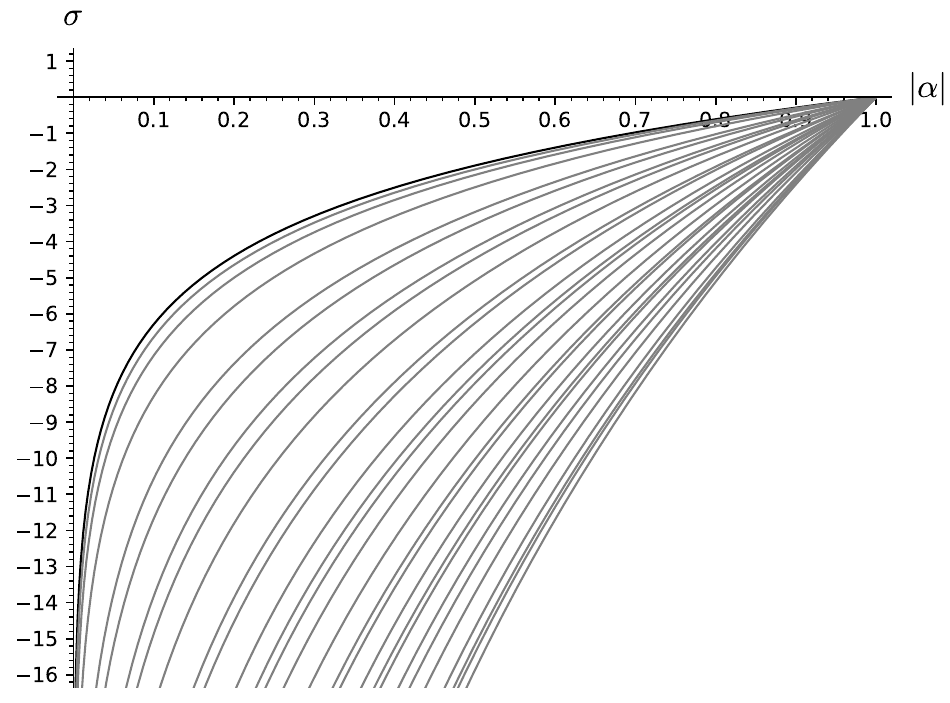}
\caption{Real part $\sigma$ of the largest $30$ poles of the $\psi$-twisted Riemann zeta function for each $0<\alpha<1$. The largest $\sigma$ (i.e., the upper curve) is associated to $p=3$, the next to $p=2$, then $p=5$ etc. The upper curve also gives the abscissa of absolute convergence $\sigma_a$.}
\label{pospoles}
\end{figure}

\subsection{$\psi$-twists of $L$-functions of Modular Forms}

Let $f$ be a normalized eigenform of weight $k$, level $N$ and character $\chi_f$. Its $L$-function is $L(s,f)= \sum_{n=1}^{\infty} c(n)  n^{-s}$, and the local factor
at a prime $p \nmid N$ is  $(1- c(p) p^{-s} + \chi(p) p^{k-1} p^{-2s})^{-1}$.

\begin{thm} Let $|\alpha|<1$.
The $\psi$-twisted $L$-function  $L(s,f,\psi)$
has an Euler product
$$ L(s,f,\psi) = \prod_{p \nmid N} ( 1 - \alpha^p c(p) p^{-s} + \alpha^{2p} \chi(p) p^{k-1} p^{-2s})^{-1} \cdot \prod_{p \mid N} ( 1 - \alpha^p c(p) p^{-s} )^{-1} .$$
If $f$ is a newform, the quadratic polynomial in the denominator for $p \nmid N$ splits as 
$(1-\alpha^p c_1(p) p^{-s}) (1-\alpha^p c_2(p) p^{-s})$, where $|c_1(p)|=|c_2(p)| = p^{(k-1)/2}$.
If $3 \nmid N$, the abscissa of absolute convergence is
$$\sigma_a = \frac{3}{\log(3)} \log(|\alpha|) + \frac{k-1}{2} .$$ 
Furthermore, $L(s,f,\psi)$ has a meromorphic continuation to $\CC$ with no zeroes.
\label{modthm}
\end{thm} 
\begin{proof} 
The absolute values $|c_1(p)|=|c_2(p)| = p^{(k-1)/2}$ are known for newforms by a Theorem of Deligne (\cite{weil} 8.2).
If $p \mid N$ then there are three possibilities: $c(p)=0$, $|c(p)| = p^{(k-1)/2}$ or $|c(p)| = p^{(k/2) - 1}$ (see \cite{deligne-serre}).
Then the assertion follows from Theorem \ref{merom}.
 \end{proof}

\begin{exm} For an elliptic curve $E$ be an elliptic curve over $\QQ$,
the associated newform $f$ has weight $2$ with Fourier coefficients $c(p) = a_p(E) = p+1-|\tilde{E}(\Bf_p)|$ for primes $p$ of good reduction. For primes of bad reduction, the coefficients are  (depending on the type of reduction) $0$, $1$ or $-1$.
The Euler factors of the Hasse-Weil $L$-function are 
$$ ( 1- c(p) p^{-s} + \chi(p)p^{1-2s} )^{-1}, $$
where $\chi$ is the trivial character modulo $N$. 
The poles of $L(s,f,\psi)$ have real parts 
$$ \frac{p}{\log(p)} \log |\alpha| + \frac{1}{2} $$
for primes $p$ of good reduction. The arithmetic information (the values $a_p(E)$) is encoded in the imaginary parts of the poles.
By looking at a larger number of elliptic curves, we have found a statistical relationship between the values of $L(s,f,\psi)$ around $s=1$ and the rank of the curve, which can be explained by the distribution of the values $a_p(E)$. 
\end{exm}

\subsection{Estimates and Asymptotic Behaviour} 

The magnitude of the twisted $L$-function can be bounded using its Euler product.

\begin{lem} Let $(u_i)_{i \in \NN}$ be a complex sequence satisfying $u_i \neq 1$ for all $i \in \NN$.
The infinite products 
$\prod_{i=1}^{\infty} (1-u_i)$ and $\prod_{i=1}^{\infty} (1-u_i)^{-1}$ 
converge absolutely to a nonzero limit if and only if the series $\sum_{i=1}^{\infty} u_i$ converges absolutely. In this case, we have
\begin{enumerate}
\item $\displaystyle \exp \left( - \sum_{i=1}^{\infty} \left|  \frac{u_i}{1-u_i} \right| \right) \leq  \prod_{i=1}^{\infty} | 1-u_i | \leq \exp \left( \sum_{i=1}^{\infty} |u_i| \right)$,
\item $\displaystyle \exp \left(- \sum_{i=1}^{\infty} |u_i| \right) \leq \prod_{i=1}^{\infty} | 1-u_i|^{-1}   \leq \exp \left( \sum_{i=1}^{\infty} \left| \frac{u_i}{1-u_i} \right| \right)$.

\end{enumerate}
\label{propest}
\end{lem}

\begin{proof} The first part on convergence is well known. One has $|1-u_i| \leq 1 + |u_i| \leq \exp(|u_i|)$. Similarly,
$|1-u_i|^{-1} = | 1 + \frac{u_i}{1-u_i} |  \leq \exp(| \frac{u_i}{1-u_i} | )$. This yields the right inequalities (1) and (2). By taking the reciprocal values, one obtains the left inequalities. This completes the proof.
\end{proof}

\begin{prop} Let $M$ be a pure motive of degree $d$ and weight $w$.
Let $L_S(M,s,\psi)$ be the twisted $L$-function with the Euler factors at bad primes removed.
For $\Re(s) > \sigma_a$, let 
$u_p=p^{\frac{p}{\log(p)} \log |\alpha| + \frac{w}{2} - \Re(s) } $. Then
\begin{equation} \exp\left( - \sum_p  d  \, u_p \right) \leq | L_S(M,s,\psi) | \leq \exp \left( \sum_p d\, \frac{u_p}{1-u_p} \right) .
\label{estformula}
\end{equation}

\label{est}
\end{prop}
\begin{proof} Each Euler factor  of $L_S(M,s,\psi)$ can be factorized into \linebreak $ \prod_{i=1}^d (1- \alpha^p c_{p,i} p^{-s})^{-1}$. Now the claim follows from  
Lemma \ref{propest} (2). 
\end{proof}

\begin{exm} Let $f$ be a newform of weight $2$.
The following table shows upper and lower bounds of  $| L(s,f,\psi)|$ for $\alpha=0.7$ as a function of $\sigma=\Re(s)$. \\

\begin{center}
\begin{tabular}{|c|c|c|}
\hline
$\sigma$ & $ \exp\left( - \sum_p  2  \, u_p \right)  $ & $ \exp \left( \sum_p 2\, \frac{u_p}{1-u_p} \right) $ \\
\hline
1 & 0.2670 & 6.0508 \\
\hline
2 & 0.5951 & 1.8248 \\
\hline
3 & 0.7988 & 1.2739 \\
\hline
4 & 0.9024 & 1.1126 \\
\hline
5 & 0.9527 & 1.0507 \\
\hline
6 & 0.9769 & 1.0239 \\
\hline
7 & 0.9887 & 1.0115 \\
\hline
8 & 0.9944 & 1.0056 \\
\hline
9 & 0.9972 & 1.0028 \\
\hline
10 & 0.9986 & 1.0014 \\
\hline
\end{tabular}

\end{center}
\label{magn}
\end{exm}
\vspace*{5mm}

\begin{rem*}
A similar bound holds for the tail of the Euler product, where the Euler factors at all primes $p < X$ are removed. Then the upper and lower bounds for $| L_S(s,f,\psi)|$ are very close to $1$,  showing that for $s$ away from the boundary of convergence, the value is dominated by the first few Euler factors. The corresponding coefficients
$(c(p))_{p \leq X}$ form a {\em signature} (see \cite{dabrowski}  for elliptic curves).
Now we see that the signature determines  $L(s,f,\psi)$ up to controlled factor close to $1$. 
\end{rem*}

\section{Twisting $p$-adic Dirichlet series and Euler products}
In this section, we show that $\psi$-twisting solves a fundamental convergence problem for $p$-adic Dirichlet series, yielding a class of genuine $p$-adic Euler products associated with classical arithmetic functions.  

\subsection{$p$-adic Dirichlet Series}
Let $p$ be an odd prime number. For any $p$-adic unit $a \in \ZZ_p^{\times}$, the Teichmüller character $\omega(a)$ is the unique $(p-1)$-th root of unity satisfying 
 $a \equiv \omega(a) \mod p$. Then we have $a = \omega(a)  \langle a \rangle$, where $ \langle a \rangle \in 1+p\ZZ_p$.
The function $ s \mapsto \langle a \rangle^s = \exp(s \log_p \langle a \rangle)$ is a well defined analytic function for $s \in \CC_p$ in the disk $|s|_p< p^{(p-2)/(p-1)}$.

\begin{dfn} Let $c(n)$ be an arithmetic function with values in $\CC_p$. A $p$-adic Dirichlet series is a series of the form
$$  L_p(s,c) = \sum^{\infty}_{\substack{n=1 \\ p\, \nmid\, n}} c(n) \langle n \rangle^{-s}.  $$
\label{dirichletseries}
\end{dfn}
Unlike the complex case, convergence is straightforward.
Since $|\langle n \rangle^{-s}|_p = 1$ for all $s$ within the disk where  $\langle n \rangle^{-s}$ is defined,
the Dirichlet series converges if and only if the coefficients converge to $0$.

\begin{prop} The $p$-adic Dirichlet series $\displaystyle\sum^{\infty}_{\substack{n=1 \\ p\, \nmid\, n}} c(n) \langle n \rangle^{-s}$ converges in the disk $|s|_p<p^{(p-2)/(p-1)}$
if $\lim_{n \rightarrow \infty} |c(n)|_p = 0$. Otherwise, it diverges for all $s$.
\label{dirichletconv}
\end{prop}

This presents a significant obstacle. For most classical arithmetic functions, like Dirichlet characters $\chi$, the coefficients satisfy $|c(n)|_p = 1$ and the corresponding Dirichlet series diverges. In fact, $p$-adic $L$-functions are constructed by interpolating special values of the complex $L$-function. There had been some progress regarding Dirichlet series expansions of $p$-adic $L$-functions (see \cite{delbourgo2006,delbourgo2009, KnWa,zhao}), but these expansions are limits of certain partial sums and not $p$-adic Dirichlet series. 

We will show in Section \ref{psipadic} that $\psi$-twisting produces
a convergent $p$-adic Dirichlet series which has desirable analytic properties, including an Euler product expansion.

 \subsection{$p$-adic Euler Products and Analyticity }
We show that a convergent $p$-adic Dirichlet series  is an analytic function in $s$ and admits an Euler product if the coefficients are multiplicative.

 \begin{prop} If $\lim_{n \rightarrow \infty} |c(n)|_p = 0$, the function $L_p(s,c)$ is analytic in the disk 
 $|s|_p<p^{(p-2)/(p-1)}$ and its Mahler expansion is given by
 $$  L_p(s,c)
 = \sum_{n=0}^{\infty} \left( \sum_{\substack {a=1 \\ (a,p)=1}}^{\infty} c(a) (\langle a \rangle - 1)^n \right) \binom{-s}{n}. $$
 Furthermore, if the coefficients $c(n)$ are multiplicative, the series admits a convergent $p$-adic Euler product
 $$  L_p(s,c) =
\prod_{ l \neq p}  \left(1 + c(l) \langle l \rangle^{-s} + c(l^2) \langle l \rangle^{-2s} + \dots \right). $$
If the coefficients are completely multiplicative then 
 $$  L_p(s,c) =
\prod_{ l \neq p}  (1- c(l) \langle l \rangle^{-s})^{-1} . $$
\label{eulerproduct}
\end{prop}
 \begin{proof} 
 We use the formula
$$  {\langle a \rangle}^{s} = \sum_{n=0}^{\infty} \binom{s}{n}  (\langle a \rangle - 1)^n $$
for $a \in \Zp^{\times}$ and let $s \in \Cp$ with $|s|_p < p^{(p-2)/(p-1)}$ (see \cite{Wa2} p. 54). Then
\begin{align} 
\label{function}
\sum^{\infty}_{\substack{a=1 \\ (a,p)=1}} c(a) {\langle a \rangle}^{-s}     & =   
\sum^{\infty}_{\substack{a=1 \\ (a,p)=1}} \sum_{n=0}^{\infty}  c(a)  \binom{-s}{n}  (\langle a \rangle - 1)^n  \\
& = \ \sum_{n=0}^{\infty} \sum_{\substack {a=1 \\ (a,p)=1}}^{\infty} c(a) (\langle a \rangle - 1)^n  \binom{-s}{n}.  \nonumber
\end{align}
We can change the order of summation  since $c(a)  \binom{-s}{n}  (\langle a \rangle - 1)^n$ converges 
to $0$ uniformly in $a$ and $n$ for each $s$. Since $|\langle a \rangle - 1|_p < \frac{1}{p}$ and $|c(a)|_p \leq C$ for all $a$ with $p\nmid a$ for some constant $C$, the inner sum satisfies
$$ \left | \sum_{\substack {a=1 \\ (a,p)=1}}^{\infty} c(a) (\langle a \rangle - 1)^n \right|_p \leq  \frac{C}{p^n} . $$
Therefore, the $p$-adic function (\ref{function}) is analytic and the radius of convergence is (at least) $p^{(p-2)/(p-1)}$ (see \cite{Wa2} 5.8).

For the Euler product, 
 let $X>0$ and consider the finite product
\begin{equation}
 P_X (s) = \prod_{\substack{ l \neq p \\ l \leq X}}  \left(1 + c(l) \langle l \rangle^{-s} + c(l^2) \langle l \rangle^{-2s} + \dots \right)
 \label{finprod}
 \end{equation}
 over primes $l$ with $l \neq p$ and $l \leq X$. 
 Since $\lim_{n \rightarrow \infty} |c(n)|_p = 0$, 
each factor of (\ref{finprod}) converges. The terms in $P_X (s)$ can be rearranged and, using the multiplicativity of the coefficients, one obtains the series
$$ P_X(s)=  \sum^{\infty}_{\substack{n=1 \\ p\, \nmid\, n \\ l\mid n \Rightarrow l \leq X}} c(n) \langle n \rangle^{-s}$$
over positive integers $n$ which are not divisible by $p$ and divisible only by primes $\leq X$.  As $X \rightarrow \infty$,  the difference $L_p(s,c) - P_X(s)$ tends to zero because the coefficients $c(n)$ converge to zero as $n \rightarrow \infty$.
This establishes the convergence of the Euler product.  If $c(n)$ is completely multiplicative, then each factor of (\ref{finprod}) is a convergent geometric series. This  completes the proof. 
 \end{proof}

 \subsection{Twisting $p$-adic Dirichlet series}
\label{psipadic}

We now apply  the $\psi$-twist, choosing a parameter $\alpha$ such that $|\alpha|_p <1$, e.g., $\alpha=p$.

\begin{dfn} The  $\psi$-twisted $p$-adic Dirichlet series is defined as
$$ L_p(s,c,\psi) = \sum^{\infty}_{\substack{n=1 \\ (n,p)=1}} \psi(n) c(n) {\langle n \rangle}^{-s} = \sum^{\infty}_{\substack{n=1 \\ (n,p)=1}} \alpha^{S(n)} c(n) {\langle n \rangle}^{-s} .$$
\end{dfn}

The following proposition shows that the twist acts as a convergence factor.

\begin{prop} Let $|\alpha|_p<1$. If the coefficients $|c(n)|_p$ are bounded, the $\psi$-twisted $p$-adic Dirichlet series $L_p(s,c,\psi)$ converges for  $|s|_p < p^{(p-2)/(p-1)}$.
\label{padicconv}
\end{prop}
\begin{proof} The $p$-adic absolute value of the new coefficients is $| \alpha^{S(n)} c(n) |_p$. Since $|c(n)|_p$ is bounded, $|\alpha|_p <1$, and $S(n) \rightarrow \infty$ as $n \rightarrow \infty$, the twisted coefficients converge to $0$. Then the result follows from Proposition \ref{dirichletconv}.
\end{proof}

Combining Propositions \ref{eulerproduct} and \ref{padicconv}, we arrive at the main theorem of this section.

\begin{thm} Let the coefficients $|c(n)|_p$ be bounded and let $|\alpha|_p<1$. The $\psi$-twisted $p$-adic Dirichlet series $L_p(s,c,\psi)$ 
defines an analytic function  for $|s|_p<p^{(p-2)/(p-1)}$, with Mahler expansion
$$ L_p(s,c,\psi) = \sum_{n=0}^{\infty} \left( \sum_{\substack {a=1 \\ (a,p)=1}}^{\infty} \alpha^{S(a)} c(a) (\langle a \rangle - 1)^n \right) \binom{-s}{n} .$$
If $c(n)$ is completely multiplicative, the function has a convergent Euler product
$$  L_p(s,c) =
\prod_{ l \neq p}  (1- \alpha^l c(l) \langle l \rangle^{-s})^{-1} . $$
\label{analytic}
\end{thm}
\begin{proof} Convergence is established by the preceding proposition. The Mahler expansion and the existence of the Euler product then follow from applying 
Proposition \ref{eulerproduct} to the twisted coefficients $\psi(n) c(n)$.
\end{proof}

\begin{rem} Can  the sequence of $\psi$-twisted Dirichlet series converge to a classical $p$-adic $L$-function (or to any other function) as $|\alpha|_p$ approaches $1$\,? In contrast to the complex case (see Propositions \ref{limit} and \ref{limitcrit}), this cannot be true. In fact, each coefficient of the Mahler expansion would have to converge as $|\alpha|_p \rightarrow 1-$.
But for $|\alpha|_p>|\beta|_p$, the strong triangle inequality gives $| \alpha - \beta|_p = |\alpha|_p$. Therefore, the Mahler coefficients cannot converge as $|\alpha|_p$ tends to $1$. 
\end{rem}

\begin{cor} The special values of the $\psi$-twisted Dirichlet series are:
\begin{align*}
L_p(-1,c,\psi)=  & \sum^{\infty}_{\substack{n=1 \\ (n,p)=1}} \alpha^{S(n)} c(n)  \langle n \rangle \\
L_p(0,c,\psi) =& \sum^{\infty}_{\substack{n=1 \\ (n,p)=1}} \alpha^{S(n)} c(n)  \\
L_p(1,c,\psi) =&  \sum_{n=0}^{\infty} (-1)^n  \sum_{\substack {a=1 \\ (a,p)=1}}^{\infty} \alpha^{S(a)} c(a) (\langle a \rangle - 1)^n   \\
L_p(2,c,\psi) =&  \sum_{n=0}^{\infty} (-1)^n (n+1) \sum_{\substack {a=1 \\ (a,p)=1}}^{\infty} \alpha^{S(a)} c(a) (\langle a \rangle - 1)^n 
\end{align*}
\end{cor}

\begin{cor} For a Dirichlet character $\chi$  and $|\alpha|_p<1$, the twisted series $L_p(s,\psi \chi)$ is an analytic function for $|s|_p<p^{(p-2)/(p-1)}$ with the Euler product
$$  L_p(s,\psi \chi) = \sum^{\infty}_{\substack{n=1 \\ (n,p)=1}} \psi (n) \chi (n)  {\langle n \rangle}^{-s}   
= \prod_{\substack{ l \neq p \\  \text{\tiny $l$ prime}}} (1- \alpha^l \chi (l) \langle l \rangle^{-s})^{-1} . $$

\label{padiceuler}
\end{cor}

\begin{rem} This construction provides a class of genuine $p$-adic Dirichlet series and Euler products associated to classical arithmetic objects like Dirichlet characters.
The functions
$$ \xi_n (s) = \alpha^{S(n)} \langle n \rangle ^{-s} $$
where $n \in \NN$ and $p \nmid n$, form a completely multiplicative system of analytic functions that converge to $0$ as $n \rightarrow \infty$, providing a 
 concrete realisation of the abstract space of {\em shadow $\nabla$-functions} considered by Delbourgo  (see \cite{delbourgo2009}). Our future work will explore the relationship between these new analytic functions and number-theoretic properties.
 \end{rem}

 \begin{exm} In a similar way as the Riemann zeta function, the expansion of the $p$-adic analogue is related to the values of many classical arithmetic functions. The following table contains the Dirichlet coefficients $c(n)$ in the expansion $\sum_{p \nmid n} c(n) \langle n \rangle^{-s}$ of certain products, quotients and translations of the twisted $p$-adic zeta functions $L_p(s,\psi \omega^k)$ associated to the Teichmuller characters $\omega^k$, where $k=0,\dots,p-2$. As usual, $d(n)$ denotes the number of divisors, $\tau(n)$ the sum of divisors, $\sigma_k(n)$  the sum of $k$-th powers of divisors of $n$ and $\phi(n)$ Euler's totient function.\\
 
 \begin{tabular}{|l|c|}
  \hline
 $p$-adic Dirichlet series & $n$-th Dirichlet coefficient, $p\nmid n$ \\
 \hline
 $L_p(s, \psi)^2$ & $\alpha^{S(n)} d(n)$ \\
  \hline
 $L_p(s,\psi) L_p(s-1,\psi \omega)$ & $\alpha^{S(n)} \tau(n)$ \\
  \hline
  $L_p(s,\psi) L_p(s-k,\psi \omega^k)$ & $\alpha^{S(n)} \sigma_k(n)$ \\
   \hline
  $L_p(s-1,\psi \omega) / L_p(s,\psi)$ & $\alpha^{S(n)} \phi(n)$ \\
   \hline
  \end{tabular}\\
 
 This table can be continued with similar relationships. The proof is along the same lines as in the complex case with small adjustments. Note that $\omega^k(n) \langle n \rangle^{-(s-k)} = n^k  \langle n \rangle^{-s}$ for $p \nmid n$. Furthermore, we use the relation $\alpha^{S(d)} \alpha^{S(n/d)} = \alpha^{S(n)}$ for $d \mid n$.
 
 \end{exm}

\bigskip

\noindent {\em Acknowledgments.} The authors thank Jerzy Kaczorowski for his valuable support. We are also indebted to Daniel Delbourgo for helpful discussions. We would also like to thank Christopher Deninger for his comments. \\
The authors acknowledge the use of Gemini 3.8 Flash (accessed September 2026) for assistance in exploring the relevant literature on partial Euler products and 
refining the technical presentation of the proofs. 
The authors assume sole responsibility for the contents of this work.

\bigskip

\end{document}